\documentclass[a4paper,12pt, english]{article}
\usepackage[utf8x]{inputenc}
\usepackage{amsmath, amsthm, amssymb, graphicx, makeidx, mathrsfs, enumerate}
\usepackage{setspace}
\usepackage{hyperref}
\usepackage{multirow}
\usepackage[T1]{fontenc}
\usepackage{babel}
\usepackage{authblk}

\newcommand{\Z}{\mathbb Z}	
\newcommand{\Q}{\mathbb Q}	
\newcommand{\F}{\mbox{$\mathbb F$}}	
\newcommand{\T}{\xi}
\newcommand{\K}{L}

\usepackage{vmargin}
\setmarginsrb     { 1.0in}  
                        { 0.9in}  
                        { 0.9in}  
                        { 1.0in}  
                        {  20pt}  
                        {0.25in}  
                        {   9pt}  
                        { 0.3in}  
\newtheorem{theorem}{Theorem}[section]

\newtheorem{lemma}[theorem]{Lemma}

\newtheorem{corollary}[theorem]{Corollary}

\newcommand{\rs}{u}

\usepackage{etoolbox}
\patchcmd{\thebibliography}
  {\settowidth}
  {\setlength{\parsep}{0pt}\setlength{\itemsep}{0pt plus 0.1pt}\settowidth}
  {}{}

\usepackage{blindtext}
\usepackage{sectsty}
\sectionfont{\centering}

\makeindex
\begin{document}

\pagenumbering{arabic}
\author[]{Anuj Jakhar}

\author[]{{\small {ANUJ JAKHAR\footnote{The Institute of Mathematical Sciences, HBNI, CIT Campus, Taramani, Chennai - 600113, Tamil Nadu, India. Email : anujjakhar@iisermohali.ac.in}  ~~ SUDESH K. KHANDUJA\footnote{Corresponding author.} \footnote{Indian Institute of Science Education and Research Mohali, Sector 81, Knowledge City, SAS Nagar, Punjab - 140306, India  \textsc{\&} Department of Mathematics, Panjab University, Chandigarh - 160014, India. Email : skhanduja@iisermohali.ac.in} ~~AND~ NEERAJ SANGWAN\footnote{Indian Institute of Technology (IIT), Bombay,  Mumbai-400076, India, neerajsan@iisermohali.ac.in} } \\}}

\date{}
\renewcommand\Authands{}

\title{\textsc{On the discriminant of pure number fields}}

\maketitle
\begin{center}
{\large{\bf {\textsc{Abstract}}}}
\end{center}  Let $K=\mathbb{Q}(\sqrt[n]{a})$ be an extension of degree $n$ of the field $\Q$ of rational numbers, where the integer $a$ is such that for each prime $p$ dividing $n$ either $p\nmid a$ or the highest power of $p$ dividing $a$ is coprime to $p$; this condition is clearly satisfied when $a, n$ are coprime or $a$ is squarefree. The paper contains an explicit formula for the discriminant of $K$ involving only the prime powers dividing $a,n$. 

\bigskip


\bigskip

\newpage
\section{\textsc{Introduction}}
Discriminant is one of the most basic invariants associated to an algebraic number field. The problem of its computation specially for pure 
algebraic number fields has attracted the attention of many mathematicians (cf.   \cite{Ded}, \cite{Fun}, \cite{HN}, \cite{JKS27}, \cite{Lan}, \cite{Wes}). By a pure number field we mean an algebraic number field of the type $\Q(\sqrt[n]{a})$, where the polynomial $x^n-a$ with integer coefficients is irreducible over the field $\Q$ of rational numbers.   In 1897, Landsberg \cite{Lan} gave a formula for the  discriminant of pure prime degree number fields. In 1984, Funakura \cite{Fun} provided a formula for the discriminant of all pure quartic fields. In  2015, Hameed and Nakahara  \cite{HN} found a formula for the discriminant of all those pure octic fields $\Q(\sqrt[8]{a})$, where $a$ is a squarefree integer. In 2017, we gave a formula for the discriminant of pure number fields having squarefree degree (cf. \cite{JKS27}).  
  In the present  paper, our aim  is to give a  formula for the discriminant of $n$-th degree fields of the type 
$\Q(\sqrt[n]{a})$ in terms of prime powers dividing $a, n$, where for each prime $p$ dividing $n$ either $p$ does not divide $a$ or the highest power of 
$p$ dividing $a$ (to be denoted by $v_p(a)$) is coprime with $p$. With this hypothesis, a formula for the discriminant of $\Q(\sqrt[n]{a})$ is given by Gassert in $\cite{Gass}$ using a method proposed by Montes and developed in (\cite{JG1}-\cite{JG4}), but our proof given here is based on the classical Theorem of Ore about Newton polygons and is more or less self-contained. 

Precisely stated, we prove:
\begin{theorem}\label{1.1}
Let $K=\mathbb Q(\theta)$ be an algebraic number field  with discriminant $d_K$,  where $\theta$ is a root of an irreducible polynomial  $f(x) = x^{n} - a$ belonging to $\mathbb Z[x]$.  Let $\prod\limits_{i=1}^k p_i ^{s_i}, \prod\limits_{j=1}^l q_{j}^{t_j}$ be the prime factorizations of $n,|a|$ respectively. Let $m_j, n_i$ and $r_i$ stand respectively for the integers $\gcd(n, t_j), \frac{n}{p_i^{s_i}}$ and $ v_{p_i} (a^{p_i-1} -1)-1$. Assume that  for each $i$,  either $v_{p_i}(a)=0 $ or $v_{p_i}(a)$ is coprime to $p_i$.  Then $$d_K 
=(-1)^{\frac{(n-1)(n-2)}{2}} sgn(a^{n-1})(\prod\limits _{i=1}^k p_i ^{ v_i}) 
 \prod\limits_{j=1}^l q_j ^{n-m_j}, 
$$ where   $v_i $ equals  $ns_i -2n_i \sum\limits_{j=1}^{\min\{r_i,s_i\} }  p_i^{s_i 
-j}$ or $ns_i $ according as  $r_i>0$ or not. 
\end{theorem}
The following corollary is an immediate application of the above theorem.
\begin{corollary}\label{2.22}
Let $p$ be a prime number and $a\neq \pm 1$ be a squarefree integer. Let $K=\Q(\theta)$ with $\theta$ a root of $x^{p^s}-a$. If $r$ stands for the integer $v_p(a^{p-1}-1)-1$, then $d_K$ is $(-1)^{\frac{(p^s-1)(p^s-2)}{2}}p^{\nu}a^{p^s -1}$, where $\nu$ equals $sp^s- 2\sum\limits_{j=1}^{\min\{r,s\}} p^{s-j}$ or $sp^s$ according as $r > 0$ or not.
\end{corollary}

It can be easily seen that in the special case when $p= 2$ and $s=3$, the formula obtained  in the above corollary for $K=\Q(\theta)$  can be restated in the following form as given in \cite{HN}: \\ $d_K = \left\{
              \begin{array}{llll}
         -2^{24} a^{7}   , & \hbox{if~ $a \equiv 2,3 (mod ~4)$,} \\
        -2^{16} a^{7}   , & \hbox{if~ $a \equiv 5,13 (mod ~16)$,} \\
     -2^{12} a^{7}   , & \hbox{if~ $a \equiv 9 (mod ~16)$,} \\
     -2^{10} a^{7}  , & \hbox{if~ $a \equiv 1 (mod ~16)$.} 
              \end{array}
            \right.$         \\   
 

The following corollary which is partially proved in \cite{Gass} will be quickly deduced from Theorem \ref{1.1}.
\begin{corollary}\label{1.5} Let $K=\Q(\theta)$ be an extension of $\Q$ with $\theta$ satisfying an irreducible polynomial $x^n -a$ over $\Z$. Then $\{1,\theta, \cdots, \theta^{n-1} \}$ is an integral basis of $K$ if and only if $a$ is squarefree and for each prime $p$ dividing $n$, $p^2\nmid (a^{p-1}-1)$.
	
\end{corollary}
\section{Preliminary Results.}

If $n = \prod\limits_{i=1}^k p_i ^{s_i}, |a| = \prod\limits_{j=1}^l q_{j}^{t_j}, f(x) = x^n - a, \theta, d_K$ are as in Theorem \ref{1.1}, $A_K$ denotes the ring of algebraic integers of $K$ and $N_{K/\Q}$ stands for the norm map, then by a basic result (\cite[Propositions 2.9, 2.13]{Nar}), we have
$$d_K[A_K : \Z[\theta]]^2  = (-1)^{\binom{n}{2}} N_{K/\mathbb{Q}}(f'(\theta)) = (-1)^{\binom{n}{2}} N_{K/\mathbb{Q}}(n\theta^{n-1}) = (-1)^{\frac{(n-1)(n-2)}{2}} n^{n}a ^{n-1}.$$ So $d_K$ is determined as soon as the exact power of each $p_i, q_j$ which divides $[A_K : \Z[\theta]]$ is known. We first deal with the primes $q_j$ dividing $a$ because these are easier to handle. For primes $p_i$ dividing $n$ and not dividing $a$, $v_{p_i}([A_K : \Z[\theta]])$ is obtained essentially in two stages. The first stage deals with the situation when $n$ is a prime power. Then we establish a relation between $v_{p_{i}}([A_K : \Z[\theta]])$ and $v_{p_{i}}([A_{K_{i}} : \Z[\theta_{i}]])$, where $A_{K_i}$ is the ring of algebraic integers of $K_i = \Q(\theta_i)$ with $\theta_i = \theta^{n_i}$ satisfying the polynomial $x^{p_i^{s_i}} - a$. For calculating the prime powers dividing these indices, we use a particular case of the Theorem of Index of Ore (stated as Theorem 2.B) for which we need the notion of Newton polygon introduced below.

Throughout $\Z_{p}$ denotes the ring of $p$-adic integers and $\F_{p}$ the field with $p$ elements. For $c$ in $\Z_{p}$, $v_p(c)$ stands for the $p$-adic valuation of $c$ defined by $v_p(p) = 1$ and $\bar{c}$ for the image of $c$ under the canonical homomorphism from $\Z_{p}$ onto $\F_{p}$.\\


\noindent  {\bf Definition.} Let $p$ be a prime number and $g(x) = \sum\limits_{j=0}^{n}c_jx^{j}$ be a polynomial over $\Z_p$ with $c_0c_n \neq 0 $. To each non-zero term $c_ix^i$, we associate a point $(n -i, v_p(c_i))$  and form the set \vspace*{-5mm}$$P = \{(j, v_p(c_{n-j})) | 0\leq j\leq n, ~c_{n-j}\neq 0\}.\vspace*{-3mm}$$ The Newton polygon of $g(x)$ with respect to $p$ (also called the $p$-Newton polygon of $g(x)$) is the polygonal path formed by the lower edges along the convex hull of points of $P$. Note that the slopes of the edges are increasing when calculated from left to right.\\

\noindent  {\bf Definition.} Let $g(x) = x^n + a_{n-1}x^{n-1} + \cdots + a_0$ be a polynomial over $\Z_p$ such that the $p$-Newton polygon of $g(x)$ consists of a single edge having positive slope $\lambda$, i.e.,  $\min\{\frac{v_p(a_{n-i})}{i}| 1\leq i\leq n\} = \frac{v_p(a_0)}{n} = \lambda$. Let $e$ denote the smallest positive integer such that $e\lambda \in \Z$. We associate with $g(x)$ a polynomial $T(Y) \in \F_p[Y]$ not divisible by $Y$ of degree $\frac{n}{e} = t$ (say) defined by    $$\vspace*{-0.1in} T(Y) = Y^t + \sum\limits_{j=1}^{t} \overline{\left({\frac{a_{n-ej}}{p^{ej\lambda}}}\right)}Y^{t-j}.\vspace*{-0.1in}$$ To be more precise,  $T(Y)$ will be called the polynomial associated with $g(x)$ with respect to $p$.

\noindent \textbf{Example.} Let $g(x) = (x+5)^4 - 5$. One can easily check that the $2$-Newton polygon of $g(x)$ consists of only one edge with slope $\lambda$ = $1/2$. With notations as in the above definition, we see that $e = 2, n = 4, t =2$ and the polynomial associated with $g(x)$ with respect to $2$ is $T(Y) = Y^2 + Y + \bar{1}$ belonging to $\F_2[Y]$.


We shall use the following weaker versions of the two theorems proved by Ore in a more general set up (cf. \cite{Ore}, \cite[pp. 322-325]{Mo-Na}, \cite[Theorem 1.1]{SKS1}). Their proofs are omitted.

 \noindent\textbf{Theorem 2.A.} 
{\it Let $p$ be a prime number. Let $g(x)$ = $\displaystyle \sum_{i=0}^{n}a_ix^i,~a_0\neq 0$ belonging to $\Z[x]$ be a monic polynomial  such that $g(x)\equiv x^n~(mod~ p)$. Suppose that the $p$-Newton polygon of $g(x)$ consists of $k$ edges $S_1,  \cdots, S_k $ having positive slopes $\lambda_1< \cdots < \lambda_k.$ Let $l_i$ denote the length of the horizontal projection of $S_i$ and $e_i$ be the smallest positive integer such that $e_i\lambda_i  \in \Z$.  Then $g(x)=g_1(x) \cdots g_k(x),$ where $g_i(x)$ is a monic polynomial over  $\Z_p$  of degree $l_i$ whose $p$-Newton polygon has a single edge which is a translate of $S_i$.}

 \bigskip
 
\noindent\textbf{Theorem 2.B.} {\it Let  $p$, $g(x)$, $S_i$, $\lambda_i$, $e_i$, $l_i$ and $g_i(x)$ be as in the above theorem for $1\leq i\leq k$. Let $T_i(Y)$ denote the polynomial associated with $g_i(x)$ with respect to $p$. Assume that $g(x)$ is irreducible over $\Q$. Let $\beta$ be a root of $g(x)$ and $K = \Q(\beta)$. If $T_i(Y)$ is a product of distinct monic irreducible polynomials over $\F_p$ for each $i$, then the highest power of $p$ dividing the index $[A_K : \Z[\beta]]$ equals the number of points with positive integer coordinates lying on or below the $p$-Newton polygon of $g(x)$ away from the vertical line passing through the last vertex of this polygon.}

 \bigskip


 The following basic lemma to be used in the sequel is already known (cf. \cite[Problem 435]{Book}). We omit its proof. As usual for a real number $\lambda$, $\lfloor\lambda \rfloor$ stands for the greatest  integer not exceeding $\lambda$.\\
 
 \noindent\textbf{Lemma 2.C.} {\it Let $t,  n$ be positive integers with $\gcd(t, n) = m$. Let $P$ denote the set of points in the plane with positive integer coordinates lying inside or on the triangle with vertices $(0,0), (n, 0), (n,t)$ which do not lie on the line $x = n$.  Then
 $$\# P = \sum\limits_{i=1}^{n-1} \bigg\lfloor \dfrac{it}{n} \bigg\rfloor  = \dfrac{1}{2} [(n-1)(t-1) + m -1].$$}

With notations as in Theorem \ref{1.1}, using the above lemma and Theorem 2.B we now prove the following result which determines $v_{q_j}([A_K : \Z[\theta]])$. 
\begin{lemma}\label{2.1}
Let $f(x)=x^n-a, K = \Q(\theta)$, $|a| = \prod\limits_{j=1}^{l}q_j^{t_j}$ and $m_j=\gcd(n, t_j)$ be as in Theorem \ref{1.1}. For a fixed prime $q_j$ dividing $a$, suppose that $q_j$ does not divide $m_j$. Then  $v_{q_j}([A_K : \Z[\theta]])$ = $\frac{1}{2}[(n-1)(t_j-1) + m_j -1]$.
\end{lemma}
\begin{proof}
 Clearly the $q_j$-Newton polygon of $f(x)$ consists of single edge having slope 
$\frac{t_j}{n}$. It can be easily seen that the polynomial associated with $f(x)$ with respect to $q_j$ is  $T(Y) = Y^{m_j} -  \overline{( \frac{a}{q_j^{t_j}}) }$ belonging to $\F_{q_j}[Y]$. By hypothesis $q_j \nmid m_j$. So $T(Y)$ has no repeated roots. The desired equality now follows immediately from Theorem 2.B and Lemma 2.C.  
\end{proof}

\indent The following simple result is well known. For reader's convenience, we prove it here.\\
\textbf{Lemma 2.D.}
 {\it For any positive integer $\rs\leq p^s$ with  $p$ a prime and $s>0$ an integer, one has $v_p (\binom{p^s}{\rs}) = s-v_p(\rs)$.}
\begin{proof}
Using the fact that for any natural number $m$, $v_p(m!) = \sum\limits_{j=1}^{\infty}\lfloor \frac{m}{p^j} \rfloor$, we see that 
     \begin{equation}\label{eq:2.D1}
     v_p \left(\binom{p^s}{\rs}\right) = v_p (p^s!) -v_p (\rs!) -v_p((p^s -\rs)!) =\sum\limits_{j=1}^{s} p^{s-j} - \sum\limits_{j=1}^{s} \bigg\lfloor \frac{\rs}{p^j} \bigg\rfloor -\sum\limits_{j=1}^{s} \bigg\lfloor p^{s-j} - \frac{\rs}{p^j} \bigg\rfloor.
\end{equation}         Keeping in mind that $\big\lfloor p^{s-j} - \frac{\rs}{p^j} \big\rfloor = p^{s-j} - \frac{\rs}{p^j}$ or $p^{s-j} - \big\lfloor\frac{\rs}{p^j}\big\rfloor -1$ according as $j \leq v_{p}(\rs)$ or not, the desired equality follows immediately from $(\ref{eq:2.D1})$.
\end{proof}  

%
  
Using the above lemma and Theorems 2.A, 2.B, we prove the following result which plays a significant role in the proof of Theorem \ref{1.1}.
\begin{lemma} \label{2.3}
Let $\K = \Q(\alpha)$  be an algebraic number field with  $\alpha$ a root of an irreducible polynomial $x^{p^s} - a$ belonging to $\Z[x]$ where $p$ is a prime number not dividing $a$ and $s$ is a positive integer. Let $A_{\K}$ be the ring of algebraic integers of $\K$. If $r$ stands for the integer $v_{p}(a^{p-1}-1) - 1$, then the exact power of $p$ dividing the index $[A_{\K} : \Z[\alpha]]$ is $\sum\limits_{i=1}^{\min\{r,s\}} p^{s-i}$ or $0$ according as $r$ is positive or not. 
\end{lemma}
\begin{proof} Since $p\nmid a$, $p$ divides $a^{p-1}-1$ and hence $r\geq 0$.
Set $\T = \alpha - a$, so that $\T$ is a root of $g(x) = (x+a)^{p^s} - a$ and $\Z[\T] = \Z[\alpha]$. Observe that $v_p(a^{p^s-1}-1) = v_p(a^{p-1}-1)$ which can be quickly verified keeping in mind that $p^s - 1 = (p-1)m$ with $m \equiv 1~(mod~p)$ and $a^{p-1} \equiv 1~(mod~p).$  When $r = 0$, then the lemma is trivially true because $g(x)$ is an Eisenstein polynomial with respect to $p$  in this situation by virtue  of the above observation and $p$ does not divide $[A_L : \Z[\xi]]$ in view of a basic result (cf. \cite[Lemma 2.17]{Nar}). From now on, it may be assumed that $r \geq 1$.
\\ \indent We first prove the lemma when $r > s$. Using Lemma 2.D, it can be easily seen that the successive vertices of the $p$-Newton polygon of $g(x) = x^{p^s} + \binom{p^s}{1}ax^{p^s -1} +\cdots +\binom{p^s}{p^s -1}a^{p^s -1}x + a^{p^s} -a$  are given by the set $
    \{  (0,0), (p^s -p^{s-1},1), (p^s -p^{s-2},2),\cdots, (p^s -1,s), (p^s ,r+1)\} 
     $.          In this case  the $p$-Newton polygon of $g(x)$ has $s+1$ edges with slopes $\lambda_i = 
    \frac{1}{p^{s-i+1}
     -p^{s-i}}$ for $1\leq i \leq s$ and $\lambda_{s+1} = r+1-s$. Applying  Theorem  2.A, we see that $g(x)=\prod\limits_{i=1}^{s+1} g_i (x)$, where $g_i(x) \in  \mathbb{Z}_p[x]$ is a monic polynomial whose $p$-Newton polygon has a single edge  with slope $\lambda_i$ and $\deg(g_i(x)) = p^{s-i+1}-p^{s-i}$ for $1\leq i\leq s$, $\deg(g_{s+1}(x)) = 1$.
         Clearly the polynomial associated with $g_i(x)$ with respect to $p$ is a monic linear polynomial in $\F_{p}[Y]$.  Note that the number of points with positive integral entries which lie on or below the $p$-Newton polygon of 
         $g(x)$ with ordinate $i$ is $p^{s-i}$ for  $1\leq i \leq s$. So it follows from Theorem 2.B that $v_p([A_L:\Z[\alpha]]) = v_p([A_L:\Z[\xi]]) =  \sum_{i=1}^{s} p^{s-i}$.
     
     Now we consider the case when $1\leq r\leq s$, $p$ odd.  Using Lemma 2.D, one can quickly verify  that the successive vertices of the $p$-Newton polygon of $g(x)$  are given by the set $
        \{ (0,0), (p^s -p^{s-1},1),(p^s -p^{s-2},2),\cdots, (p^s -p^{s-r},r), (p^s ,r+1) \} $.  
      Note that in this case the $p$-Newton polygon of $g(x)$ has $r+1$ edges with slopes $\lambda_i 
       = \frac{1}{p^{s-i+1} -p^{s-i}}$ for $1\leq i \leq r$ and $\lambda_{r+1} = \frac{1}{p^{s-r}}$. Applying Theorem 2.A, we see that 
    $g(x)$ can be written as a product $\prod\limits_{i=1}^{r+1} g_i (x)$  of monic polynomials belonging to $\mathbb{Z}_p[x]$, where the $p$-Newton polygon of $g_i(x)$ has a single edge with slope $\lambda_i$ and the polynomial associated with $g_i(x)$ with respect to $p$  is a  monic linear
    polynomial. Arguing as in the previous case, we see that $v_p([A_L:\Z[\xi]]) =  \sum_{i=1}^{r} p^{s-i}$ which 
     proves the lemma in this case. \\ 
  \indent Now we deal with the situation when $1\leq r\leq s$ and $p=2$.  One can check that the successive vertices of the $2$-Newton polygon of $g(x)=  x^{2^s} + \binom{2^s}{1}ax^{2^s -1} +\cdots +\binom{2^s}{2^s -1}a^{2^s -1}x + a^{2^s} -a$ are given by the set $ 
     \{ (0,0), (2^s -2^{s-1},1),(2^s -2^{s-2},2)\cdots, (2^s -2^{s-r+1},r-1),$ $(2^s ,r+1) \}$.
       The  $2$-Newton polygon of $g(x)$ has $r$ edges with slopes 
          $\lambda_i = \frac{1}{2^{s-i+1} -2^{s-i}}$ for $1\leq i \leq r-1$ and $\lambda_{r} = \frac{1}{2^{s-r}}$. It follows quickly from Theorem 2.A that  $g(x)=\prod\limits_{i=1}^{r} g_i (x)$   where $g_i(x)$ belonging to $\Z_{2}[x]$ is a monic polynomial which corresponds to 
       the   $i$-th edge of the  $2$-Newton polygon of $g(x)$.  
        Further the polynomial associated with $g_i(x)$ with respect to $2$ is a monic linear polynomial for $1\leq i\leq r-1$ and the polynomial associated with $g_r(x)$ with respect to $2$ is a second degree polynomial, say $T_r(Y)$ belonging to $\F_2[Y]$. Keeping in mind that the $2$-Newton polygon of $g_r(x)$ (being a translate of the last edge of the $2$-Newton polygon of $g(x)$) has lattice points $(0,0),(2^{s -r},1), (2^{s - r +1}, 2)$ on it, we conclude that $T_r(Y)=Y^2+Y+\bar{1}$.
So Theorem 2.B is applicable to $g(x)$. Since the number of points with positive integral entries which lie on or below the $2$-Newton polygon of $g(x)$ with ordinate $i$ is $2^{s-i}$, it follows that 
          $v_2([A_L:\Z[\xi]]) = \sum_{i=1}^{r} 2^{s-i}$. This completes the proof of the lemma. 
\end{proof}
\noindent\textbf{Notation 2.E.} For an algebraic number field  $K$, $A_K$ will denote its ring of algebraic integers, $d_K$ its (absolute) discriminant and for a non-zero ideal $I$ of $A_K$, $N_{K/\Q}(I)=[A_K:I]$ will denote the (absolute) norm of $I$. If $K=\Q(\alpha)$ with $\alpha$ an algebraic integer,   the index $[A_K:\Z[\alpha]]$ will be denoted by $Ind(\alpha)$. For a relative extension $L/K$ of algebraic number fields, $d_{L/K}$ will stand for the relative discriminant. We shall use the following formula (cf. \cite[Theorem 4.15]{Nar}) 
\begin{equation}\label{relative}
d_L =\pm d_K ^{[L:K]} N_{K/\Q}(d_{L/K}).
\end{equation}
 If $\{\alpha_1, \cdots, \alpha_n\}$ is a vector space basis of $L/K$, then $D_{L/K}(\alpha_1,\cdots,\alpha_n)$ will denote  the determinant of the  $n\times n$ matrix with $(i,j)$-th entry $Tr_{L/K}(\alpha_i \alpha_j)$, $Tr$ stands for the trace map. If $L=K(\alpha)$, $\alpha\in A_L$ with $g(x)$ as the minimal polynomial of $\alpha$ over $K$, then as in \cite[Proposition 2.9]{Nar}, it can be easily seen that \begin{equation}\label{DLoverK}
 D_{L/K}(1,\alpha,\cdots,\alpha^{n-1}) = (-1)^{\frac{n(n-1)}{2}} N_{L/K}(g'(\alpha)).
 \end{equation} 

 With the above notation, we prove the following lemma which extends Lemma \ref{2.3} to general $n$. 
\begin{lemma}\label{2.5}
Let $K=\Q(\theta)$ where $\theta$ is a root of an irreducible polynomial $x^n-a$ belonging to $\Z[x]$. Let $\prod\limits_{i=1}^{k} p_i ^{s_i}$ be the prime factorization of $n$.  Suppose that $p_i \nmid a $  for some $i$. If $r_i,n_i$ stand for the integers  $v_{p_i} (a^{p_i-1}-1)-1, \frac{n}{p_i ^{s_i}}$ respectively, then the exact power of $p_i$ dividing $Ind(\theta)$ is $ \sum\limits_{j=1}^{\min \left\lbrace r_i,s_i\right\rbrace  }n_i  p_i^{s_i 
    -j}$ or $0$ according as $r_i $ is positive or not.
\end{lemma}
\begin{proof}
Set $\theta _i =\theta ^{n_i}$ and $K_i=\mathbb{Q}(\theta_i)$. Note that $[K:K_i]=n_i$. By $(\ref{relative})$, we have \begin{equation}\label{eq:2501}
d_K =\pm d_{K_i} ^{n_i} N_{K_i/\Q}(d_{K/K_i}).
\end{equation} Claim is that $p_i$ does not divide $N_{K_i/\mathbb{Q}}(d_{K/K_i})$. 
Note that the minimal polynomial of $\theta $ over $K_i$ is $g(x) = x^{n_i} - \theta_i$. 
By a well known result (cf. \cite[Theorem 4.16]{Nar}), $d_{K/K_i}$ divides the ideal $N_{K/K_i}(g'(\theta))A_{K_i}$. So $N_{K_i/\Q}(d_{K/K_i})$ divides  $N_{K/\Q}(g'(\theta)) = \pm n_i^{n}a^{n_i-1}$, which proves the claim in view of the fact that $p_i\nmid n_ia$.  
 It is immediate from $(\ref{eq:2501})$ and the claim that \begin{equation}\label{eq:2602}
v_{p_i}(d_{K})= n_i v_{p_i}(d_{K_i}).   \end{equation}
 Using $(\ref{DLoverK})$, we see that
 \begin{equation}\label{eq:2603} D_{K/\Q}(1,\theta,\cdots,\theta^{n -1}) = \pm n^{n} a^{n -1}, 
 D_{K_i/\Q}(1,\theta_i,\cdots,\theta_i^{p_i ^{s_i}-1}) =\pm p_i ^{s_i p_i ^{s_i}} a ^{p_i^{s_i} -1}. 
 \end{equation} Recall that $n = n_i p_i^{s_i}$ and $p_i \nmid an_i$. So it is clear from $(\ref{eq:2603})$ that 
 \begin{equation}\label{eq:2604}
v_{p_i}(D_{K/\mathbb{Q}}(1,\theta,\cdots,\theta^{n -1}))= n_i 
v_{p_i}(D_{K_i/\mathbb{Q}}(1,\theta_i,\cdots,\theta_i^{p_i ^{s_i} -1})).   \end{equation}
Also by a basic result (cf. \cite[Proposition 2.13]{Nar}), $D_{K/\Q}(1,\theta,\cdots , \theta^{n-1}) =d_K Ind(\theta)^2$; consequently \begin{equation}\label{eq:2605}
v_{p_i}(D_{K/\mathbb{Q}}(1,\theta,\cdots,\theta^{n -1}))=  v_{p_i}(d_{K}) + 2v_{p_i}(Ind(\theta)). \end{equation}
Substituting from $(\ref{eq:2602})$ and $(\ref{eq:2604})$ in $(\ref{eq:2605})$, we get
\begin{equation}\label{eq:2606}
n_iv_{p_i}(D_{K_i/\mathbb{Q}}(1,\theta_i,\cdots,\theta_i^{p_i ^{s_i} -1}))= n_i v_{p_i}(d_{K_i}) + 2v_{p_i}(Ind(\theta)).   \end{equation}
Keeping in mind that $D_{K_i/\Q}(1,\theta_i,\cdots , \theta_i^{p_i ^{s_i}-1}) =d_{K_i} Ind(\theta_i)^2$, we conclude from  $(\ref{eq:2606})$ that $v_{p_i}(Ind(\theta))= n_i v_{p_i}(Ind(\theta_i))$. So the desired equality now follows from Lemma \ref{2.3}.
\end{proof}
\section{Proof of Theorem \ref{1.1}, Corollary \ref{1.5}.}
 Recall that
\begin{equation}\label{eq:99}
D_{K/\Q}(1, \theta, \cdots, \theta^{n-1}) = (Ind~\theta)^2  d_K =(-1)^{\binom{n}{2}} N_{K/\mathbb{Q}}(n\theta^{n-1}) = (-1)^{\frac{(n-1)(n-2)}{2}} n^{n}a ^{n-1}.
\end{equation}
So any prime dividing $Ind(\theta)$ must divide $an$.  
It follows from Lemmas \ref{2.1} and \ref{2.5} that 

 \begin{equation*} \label{eq:1102}
Ind(\theta)=\prod\limits _{i=1}^{k}p_i ^{ u_i} \prod\limits_{j=1}^{l} q_j ^{\frac{1}{2}[(n-1)(t_j-1) + m_j -1]} ,
\end{equation*} where $u_i$ equals  $n_i \sum\limits_{j=1}^{\min\{r_i,s_i\} }  p_i^{s_i 
-j}$ or $0 $ according as  $r_i>0$ or not. Substituting for $Ind(\theta)$ from the above equation and $n = \prod\limits_{i=1}^k p_i ^{s_i}, |a| = \prod\limits_{j=1}^l q_{j}^{t_j}$ in $(\ref{eq:99})$, we immediately obtain the desired formula for $d_K$.  \hspace*{12.7cm}$\square$\\

\begin{proof}[Proof of Corollary \ref{1.5}.] Assume first that $A_K = \Z[\theta]$. Suppose to the contrary that $a$ is not a squarefree integer, say $a = bc^2$ with $c\geq 2$. The equality $\theta^n = a = bc^2$ shows that $\theta^n$ divides $(bc)^n$ in $A_K$ and hence $\theta$ divides $bc$ which implies that $\frac{\theta^{n-1}}{c} = \frac{bc}{\theta}$ is an algebraic integer. Consequently $c$ will divide the index $[A_K : \Z[\theta]]$. This contradiction proves that $a$ is squarefree.

When $a$ is squarefree, writing the prime factorization of $n$ as $\prod\limits_{i=1}^k p_i ^{s_i}$ and on taking $r_i, v_i$  as in Theorem \ref{1.1}, we see that the discriminant $d_K$ of $K$ is given by $$d_K= (-1)^{\frac{(n-1)(n-2)}{2}} sgn(a^{n-1})(\prod\limits _{i=1}^k p_i ^{v_i}) |a|^{n-1}. 
	$$  As shown in equation (\ref{eq:99}), $$ D_{K/\mathbb{Q}}(1,\theta,\theta^2, \cdots, \theta^{n-1}) = (Ind(\theta))^2 d_K=(-1)^{\frac{(n-1)(n-2)}{2}} n^n a^{n-1}.$$ In view of the above two equations, $ \{1,\theta,\theta^2, \cdots, \theta^{n-1}\} $ is an integral basis of $K$ if and only if $n^n =\prod\limits _{i=1}^k p_i ^{v_i}$, i.e., if and only if $\prod\limits_{i=1}^k p_i^{ns_i} =  \prod\limits _{i=1}^k p_i ^{v_i}$. Keeping in mind the definition of $v_i$, the last equality holds if and only if  for $1\leq i\leq k$, $r_i\leq 0$, i.e., $p_i^2 \nmid (a^{p_i -1}-1).$ This completes the proof of the corollary.
\end{proof}


\noindent\textbf{Acknowledgement.} The second author is thankful to Indian National Science Academy, New Delhi for fellowship.


\end{document}